\documentclass[11pt]{article}
\usepackage{amssymb,amsmath,amsthm,amsfonts}
\usepackage[dvips]{graphicx}
\def\disp{\displaystyle}

\def\crr{\cr\noalign{\vskip2mm}}
\def\dfrac{\displaystyle\frac}

\def\dref#1{(\ref{#1})}

\theoremstyle{plain}
\newtheorem{theorem}{Theorem}[section]
\newtheorem{lemma}{Lemma}[section]

\numberwithin{equation}{section}

\theoremstyle{definition}
\newtheorem{definition}{Definition}

\newtheorem{remark}{Remark}[section]

\newcommand{\R}{{\mathbb R}}

\newcommand{\N}{{\mathbb N}}
\def\D{{\cal D}}
\def\Z{\mathbb{Z}}
\begin{document}

 \title{{\bf     On  Existence and Uniqueness of the Weak Solution of a eralized Boussinesq
       Equation with  Press  and Mixed Boundary Conditions
    }\footnote{This work was
   supported by   the National Natural Science Foundation of China, the
   National Basic Research Program of China (2011CB808002),  and the
   National Research Foundation of South Africa. }}
   \date{19 June, 2012}
   \author{Gol Kim$^{a}$   and  Bao-Zhu Guo$^{b,c}$\footnote{ The corresponding hor.
   Email: bzguo@iss.ac.cn} \\
   $^a${\it Center of Natural Sciences, University of Sciences, DPR
   Korea}\\
   $^b${\it Academy of Mathematics and Systems Science, Academia
   Sinica}\\ {\it Beijing 100190,   China}
   \\
   $^c${\it  School  of Computational and Applied Mathematics}\\
    {\it University of the Witwatersrand, Wits 2050, Johannesburg, South
    Africa}}

\maketitle
\begin{center}
\begin{abstract}
In this paper, the existence and uniqueness of weak solution for a
generalized Boussinesq  equation that couples the mass and heat
flows in a viscous incompressible fluid  is considered.  The
viscosity and the heat conductivity are assumed to depend  on the
temperature. The boundary condition on  velocity of the fluid is
non-standard where the dynamical pressure is given on some part of
the boundary, and the  temperature of the fluid is represented in a
mixed boundary condition.

 \vspace{0.3cm}


 \vspace{0.3cm}


\end{abstract}

\end{center}

\section{Introduction}
\setcounter{equation}{0}

The existence and uniqueness of the solution for the generalized
Boussinesq equation have been studied extensively by many authors,
see \cite{[3],[2],[1],[4],[5]} and the references therein. However,
the boundary conditions in these works are Dirichlet type for
velocity of the fluid.  In \cite{[16]} the blow-up and global
existence for nonlinear parabolic equations with Neumann boundary
conditions is considered. The recent works on the existence and
uniqueness of the solution of Boussinesq equation can be found in
\cite{[14],[13],[12],[11]}. \cite{[11]} studies the  existence and
uniqueness of stationary Boussinesq system
    with non-smooth mixed boundary conditions for the temperature, and
    non-smooth Dirichlet boundary condition for the velocity.
The local existence and uniqueness  of the solution for  the
boundary-value problem for the stationary Boussinesq heat and mass
transfer equations with the inhomogeneous non-standard boundary
conditions for the velocity,  and the mixed boundary conditions for
the temperature and concentration are obtained in \cite{[12]}.
\cite{[13]} investigates  the existence and uniqueness of weak
solutions of the stationary Boussinesq system with the homogeneous
Dirichlet boundary conditions for the velocity and temperature. The
 the existence and uniqueness of weak solutions of the
time-periodic solutions for a generalized Boussinesq equation  with
Neumann boundary conditions for temperature are studied. The global
regularity of the classical solutions for a 2D Boussinesq equation
with the vertical viscosity and  diffusivity has been investigated.

In this paper, we are concerned with the  existence and uniqueness of the solution of
a generalized Boussinesq equation with nonlinear diffusion of
the
velocity and temperature, where the press boundary condition on the velocity
of the fluid  and the mixed boundary condition on the temperature are
given, The system is described by the following initial-boundary conditions:
\begin{equation}\label{1}
\left\{\begin{array}{ll} \disp \dfrac{\partial z(x,t)}{\partial t}-\gamma(w(x,t))\triangle
z(x,t)+(z(x,t),\nabla)z(x,t)-\beta g w(x,t) \crr
\hspace{4cm} =f_1(x,t)-\nabla \pi(x,t) & \hbox{ on
}Q,\crr\disp
{\rm div}(z(x,t))=0 & \hbox{ on }  Q,\crr\disp
\dfrac{\partial w(x,t)}{\partial t}-{\rm div}(k(w(x,t))\nabla
w(x,t))+(z,\nabla)w=f_2(x,t) &  \hbox{ on }  Q,\crr\disp
z(x,t)_\tau=0,\; \pi(x,t)+\dfrac{1}{2}|z(x,t)|^2=v_1(x,t),\; w(x,t)=0 & \hbox{ on }
\Sigma_1, \crr\disp
z(x,t)=0, \;-k(w(x,t))\dfrac{\partial w(x,t)}{\partial n}=v_2(x,t) & \hbox{ on
}\Sigma_2, \crr\disp
z(x,0)=z_1(x), \;w(x,0)=w_0(x) &\hbox{ on }\Omega,
\end{array}\right.
\end{equation}
where $\Omega\subset \R^N (N=2,3)$ is a bounded domain with the
smooth boundary $\Gamma=\Gamma_1\cup \Gamma_2$ where
$\Gamma_1\cap\Gamma_2\neq \emptyset$, $Q=\Omega\times (0,T)$ for the
given $T>0$, $\Sigma_i=\Gamma_i\times (0,T), i=1,2$,
$\Sigma=\Gamma\times (0,T)$,  $n$ is the normal vector exterior to
$\Gamma$, $z(x,t)\in \R^N$ denotes the velocity of the fluid at
$x\in \Omega$ and time $t\in (0,T)$, $\pi(x,t)\in \R$ is the
hydrostatic pressure, $w(x,t)\in \R$ is the temperature, $f_1(x,t)$
is the external force, $g$ is the gravitational vector function,
$\gamma(\cdot)>0$ is the kinematic viscosity, $k(\cdot)>0$ is the
thermal conductivity, $\beta$ is a positive constant associated to
the coefficient of volume expansion, $f_2(x,t)$ is the heat source
strength,   and $v_1,v_2$ are prescribed functions to be determined
later. The i-th component of $(z(x,t),\nabla z(x,t))z(x,t)$ in
Cartesian coordinate is given by
$$
((z(x,t),\nabla)z(x,t))_i=\sum_{j=1}^N z_j(x,t) \dfrac{\partial
z_j(x,t)}{\partial x_j}, \; (z(x,t),\nabla)w(x,t)=\sum_{j=1}^N
z_j(x,t) \dfrac{\partial w(x,t)}{\partial x_j},
$$
and
$$
z_\tau(x,t)=z(x,t)-z_n(x,t)\cdot n, \; z_n(x,t)=(z(x,t)\cdot n).
$$
The classical Boussinesq equation where $\gamma$ and $k$ are
constants has been studied widely in literature, see for instance
\cite{[7],[6]}. Equation \dref{1} that is paid less attention
represents the physical system for which we the variation of the
fluid viscosity,  and the thermal conductivity  with temperature
cannot be ignored (see e.g., \cite{[5]} and the references therein).
Some researches on the generalized Boussinesq equation and the
Boussinesq system with nonlinear thermal diffusion are available in
\cite{[3],[2],[1],[4]} but the  boundary conditions in this works
are homogeneous.  It should be pointed out that  the variation of
the viscosity with the temperature is important in understanding the
process of the flow.  As a first step to this point, we establish,
in this paper, the the existence of non stationary solution of
Equation \dref{1} where the dynamical pressure is given on some part
of the boundary,  and the boundary condition for  temperature of the
fluid is mixed. Due to the stronger nonlinear coupling between the
equations, Equation \dref{1} is much difficult than the
 classical Boussinesq equations.

We proceed as follows. In section 2, we give some preliminary
results on Equation \dref{1}. With the  Galerkin method, we show, in
section 3 the existence
   of the global solution to the problem  \dref{1} under certain conditions on the
    temperature dependency of the viscosity and the thermal conductivity.
The assumption on non-constant gravitational field would be helpful
in some other  geophysical models. The uniqueness of the weak
solution is presented in section 4.

\section{Preliminary results }\label{Sec2}

For notation simplicity, we do not distinguish the functions defined
on the real number field $\R$ or the $N$-dimensional Euclidean space
$\R^N$, which is clear from the context. The $L^2(\Omega$ inner
product and the norm induced by the inner product are denoted by
$(\cdot,\cdot)$ and $|\cdot|$ respectively. The norm of the Sobolev
space $H^m(\Omega)$ is denoted by $\|\cdot\|_m$ which is reduced to
$L^2(\Omega)$ when $m=0$. $H^{-1}(\Omega)$ is the dual space of
$H_0^1(\Omega)$.
 We use $\D(0,T)$ to denote the $C^\infty$-functions with the compact on $(0,T)$ and
$\D(0,T)'$ the space of distribution associated with $ \D(0,T)$.

Now we introduce some function spaces as follows.
\begin{equation}\label{guo1}
\left\{\begin{array}{l} D=\{ \psi\in (C^\infty(\Omega))^N|\; {\rm
div}(\psi(x)=0 \hbox{ on
}\Omega \hbox{ and }\phi_\tau(x)=0 \hbox{ on }\Gamma_2\},\\
H=\hbox{ completion of } D \hbox {under the }(L^2(\Omega))^N-\hbox{norm }, \\
V=\hbox{ completion of } D \hbox {under the }(H^1(\Omega))^N-\hbox{norm }, \\
D_{\Gamma_1}=\{\varphi \in C^\infty(\Omega)|\; \varphi(x)=0 \hbox {
on }\Gamma_1\}, \\
\widetilde{H}=\hbox{ closure of }D_{\Gamma_1} \hbox{ in }
L^2(\Omega), \\
W=\hbox{ closure of }D_{\Gamma_1} \hbox{ in } H^1(\Omega).
\end{array}\right.
\end{equation}
The norm of $H$ is denoted by $|\cdot|$ and the norm of $V$ is denoted by $\|\cdot\|$. If there is no the anxiety confused in the notation, then we will also denote the norm of $\tilde{H}$ is by $|\cdot|$ and the norm of $W$ is denoted by $\|\cdot\|$.

Suppose that $(z,w)$ is  classical solution of \dref{1}.

Now, specially, we assume that $v_1,f_1\in L^2(0;T;(L^2(\Gamma_1))^N)$ and $ v_1,f_2\in L^2(0;T;(L^2(\Gamma_1)))$.

Multiplier the first equation of \dref{1} by $\phi\in V$,
 integrate by parts over $\Omega$ and take the boundary condition into account  to get
\begin{equation}\label{7}
\dfrac{d}{dt}(z,\psi)+a_{\gamma(w)}(z,\psi)+b(z,z,\psi)-(\beta g
w,\phi)=(f_1,\psi)+(\nu_1,\psi_n)_{\Gamma_1}, \;
\psi_n= (\psi\cdot n)n.
\end{equation}
Multiplier the second  equation of \dref{1} by $\varphi\in W$,
 integrate by parts over $\Gamma$, and take the boundary condition into account  to obtain
\begin{equation}\label{8}
\dfrac{d}{dt}(w,\varphi)+a_{k(w)}(w,\varphi)+c(z,w,\varphi)=(f_2,\varphi)+(\nu_2,\varphi)_{\Gamma_2}.
\end{equation}
Now let $\chi\in C^1[0,T]$ be a function such that $\chi(T)=0$.
Multiplier Equations \dref{7} and \dref{8} by $\chi$ respectively,
and integrate by parts to yield
\begin{equation}\label{9}
\begin{array}{l}
\disp \int_0^T[-(z(t),\phi\chi'(t))+
a_{\gamma(w)}(z(t),\phi\chi(t))+b(z(t),z(t), \phi\chi(t))\crr\disp
+(\beta g w(t),\phi\chi(t))]dt =
\int_0^T(f_1(t),\phi\chi(t))dt\crr\disp
+\int_0^T\int_{\Gamma_1}(\nu_1,\phi_n\chi(t))dsdt+(z_0,\phi)\chi(0).
\end{array}
\end{equation}
\begin{equation}\label{10}
\begin{array}{l}
\disp \int_0^T[-(w(t),\varphi\chi'(t))+
a_{k(w)}(w(t),\varphi\chi(t))+c(z(t),w(t), \phi\chi(t))\crr\disp]dt =
\int_0^T(f_2(t),\varphi\chi(t))dt
+\int_0^T\int_{\Gamma_2}(\nu_2,\varphi\chi(t))dsdt+(w_0,\varphi)\chi(0),
\end{array}
\end{equation}
where we write $z(\cdot,t)=z(t), w(\cdot,t)=w(t)$ by abuse of
notation without confusion from the context, and
$$
\begin{array}{l}
\disp (z,\phi)=\sum_{j=1}^N\int_\Omega z_j(x,\cdot)\phi_j(x)dx, \;
(w,\varphi)=\int_\Omega (w(x,\cdot)\varphi(x)dx, \crr\disp
a_{\gamma(w)}(u,z)=(\gamma(w){\rm rot}z(x,\cdot),{\rm rot}\phi), \;
b(z,z,\phi)=\int_\Omega({\rm rot}z(x,\cdot)\times
z(x,\cdot))\phi(x)dx, \crr\disp
a_{k(w)}(w,\varphi)=\sum_{j=1}^N\int_\Omega k(w)\dfrac{\partial
w(x,\cdot)}{\partial x_j}\cdot \dfrac{\partial \varphi(x)}{\partial
x_j}dx, \; c(z,w,\varphi)=\sum_{j=1}^N\int_\Omega
z_j(x,\cdot)\dfrac{\partial w(x,\cdot)}{\partial x_j}\varphi(x)dx.
\end{array}
$$

We are now in a position to state some preliminary results.
Throughout the paper, we always assume that there are positive
constants $\gamma_0, \gamma_1, k_0, k_1$ such that
\begin{equation}\label{Assum}
\gamma_0 \le \gamma(t)\le \gamma_1, \; k_0\le k(t)\le k_1, \forall\;
t\in [0,\infty).
\end{equation}

The following Lemmas \ref{Le1}-\ref{Le3} can be obtained by the
Sobolev inequalities and the compactness theorem. We can also refer
to theorem 1.1 of \cite{[17]} on page 107 and lemmas 1.2, 1.3 of
\cite{[17]} on page 109 (see also chapter 2 of \cite{[9]}). The
similar arguments can also be found in lemmas 1, 5 of \cite{[18]}.

\begin{lemma}\label{Le1}
The bilinear forms $a_{\gamma(\cdot,\cdot)}$ and
$a_{k(\cdot,\cdot)}$ are coercive over  $V$ and $W$ respectivel
That is, there exist constants $c_1,\acute{c}_1>0$ such that
$$
a_{\gamma(w)}(z,z)\ge \gamma_0 c_1\|z\|_1^2, \forall\; z\in V \hbox{
and } a_{k(w)}(w,w)\ge k_0 c_1^\prime\|w\|_1^2, \forall\;
w\in W.
$$
\end{lemma}
\begin{lemma}\label{Le2}
The trilinear form $b(\cdot,\cdot,\cdot)$ is a linear continuous
functional with respect to each variable defined on $(H^1(\Omega))^N$. That is, there exist a constant
$c_2>0$ such that
$$
|b(u,v,w)|\le c_2\|u\|\|v\|\|w\|, \forall\; u,v,w\in
(H^1(\Omega))^N.
$$
 Moreover, the following properties hold true

(i).~ $b(u,v,v)=0,  \forall \; u,v\in V$;

 (ii).~ $b(u,v,w)=-b(u,w,v), \forall\; u\in V, v,w\in
(H^1(\Omega))^N $;

(iii). If  $u_m\to u$ weakly on $V$  and  $v_m\to v$  strongly on
$H$,   then $b(u_m,v_m,w)\to b(u,v,w), \forall\; w\in V, w\in V$.
\end{lemma}
\begin{lemma}\label{Le3}
 The trilinear form $c(\cdot,\cdot,\cdot)$ is a  linear continuous functional with respect to each variable defined
on $V\times W\times W$. That is, there exist a constant $c_3>0$ such
that
$$
|c(z,w,\varphi)|\le c_3\|z\|\|w\|\|\varphi\|, \forall\; z\in V,
w,\varphi\in W.
$$
Moreover, the following properties hold true

(i).~ $c(u,w,w)=0,  \forall \; z\in V, w\in W$;

 (ii).~ $c(z,w,\varphi)=-c(z,\varphi,w), \forall\; z\in V, w,\varphi\in
W $;

(iii). If  $z_m\to z$ weakly on $V$  and  $w_m\to w$  strongly on
$L^2(\Omega)$,   then $c(z_m,w_m,\varphi)\to b(z,w,\varphi),
\forall\; z\in V, w\in L^2(\Omega), \varphi\in W$.
\end{lemma}
By considering  \dref{8},\dref{9}, we define the weak of \dref{1} such as;
\begin{definition}\label{De1}
 Let $Y\equiv Z\times W=(L^2(0,T;V)\cap
L^\infty(0,T;H))\times (L^2(0,T;W)\cap L^\infty(0,T;\tilde{H}))$.
Suppose that
$$
\begin{array}{l}
f_1\in L^2(0,T;V^*), f_2\in L^2(0,T;W^*), v_1\in
L^2(0,T;(H^{-1/2}(\Gamma_1))^N),  v_2\in
L^2(0,T;(H^{-1/2}(\Gamma_2)), \\
z_0\in H, w_0\in L^2(\Omega), g\in L^\infty(\Omega).
\end{array}
$$
The pair $\{z,w\}$ is said to be a weak solution of \dref{1} if it
satisfies
\begin{equation}\label{11}
\left\{\begin{array}{l}
\{z,w\}\in Y, z'\in L^1(0,T;V^*), w'\in L^1(0,T;W^*), \\
(z',\phi)+a_{\gamma(w)}(z,\phi)+b(z,z,\phi)+(\beta w g,\phi)=\langle
f_1,\phi\rangle+\langle v_1,\phi_n\rangle_{\Gamma_1}, \forall;
\phi\in
V, \\
(w',\varphi)+a_{k(w)}(w,\varphi)+c(z,w,\varphi) =\langle
f_2,\varphi\rangle+\langle v_2,\varphi\rangle_{\Gamma_2}, \forall;
\varphi\in W, \\z(0)=z_0, w(0)=w_0,
\end{array}\right.
\end{equation}
 where the  inner product in $(L^2(\Omega))^N$ is
 also denoted  by $(\cdot,\cdot)$ without confusion from the context, and
 that in $V^*$ and $V$ (also $H^{-1}(\Omega)$ and $H_0^1(\Omega)$)
 are denoted by $\langle \cdot,\cdot\rangle$.
\end{definition}

Next, we reformulate Equation \dref{11} in to the operator equation.
To this purpose, it is noticed  that for a fixed $\phi\in V$, the
functional $\phi (\in V)\to a_{\gamma(w)}(z,\phi)$  is linear
continuous. So there exists an $A_{\gamma} z\in V^*$ such that
\begin{equation}\label{12}
\langle A_{\gamma} z,\phi\rangle=a_{\gamma(w)}(z,\phi), \forall\;
\phi\in V.
\end{equation}
 Similarly, for fixed $u,v\in V$, $w(\in V)\to b(u,v,w)$ is a linear continuous functional on
 $V$.
 Hence  there exist a $B(u,v)\in V^*$ such that
\begin{equation}\label{13}
 \langle B(u,v),w\rangle =b(u,v,w), \forall\; w\in V.
 \end{equation}
We denote $B(u)=B(u,u)$. Define
\begin{equation}\label{14}
L_1(\phi)=\langle f_1,\phi\rangle+\langle v_1,\phi_n\rangle
_{\Gamma_1}, \forall\; \phi\in V.
\end{equation}
Then $L_1$   a linear continuous functional on $V$ and so there
exists a constant $c_4>0$ such that
\begin{equation}\label{15}
\|L_1\phi \|\le c_4\|\phi\|_V, \forall\; \phi\in V.
\end{equation}
Hence there exists a $F_1\in V^*$ such that $L_1(\phi)=\langle
F_1,\phi\rangle$ for all $\phi\in V$.

 By the operators defined above, we can write the second
equation of \dref{11} as
\begin{equation}\label{15g1}
\dfrac{dz}{dt}+A_{\gamma}z+B(z)+\beta g w=F_1.
\end{equation}
Similarly, we have
\begin{equation}\label{15g2}
 \langle A_{k} w,\varphi\rangle=a_{k(w)}(w,\varphi),\;  \langle C(z,w),\varphi\rangle=c(z,w,\varphi),
A_k w, C(z,w)\in W^*, \forall\; \varphi\in W.
\end{equation}
Define
\begin{equation}\label{15g3}
L_2(\varphi)=\langle f_2,\varphi\rangle+\langle v_2,\varphi\rangle
_{\Gamma_2}.
\end{equation}
Then $L_2$   a linear continuous functional defined on $W$ and so
there exists a constant $c_5>0$ such that
\begin{equation}\label{15g4}
\|L_2\varphi \|\le c_5\|\varphi\|_W, \forall\; \varphi\in W.
\end{equation}
Hence there exists a $F_2\in W^*$ such that $L_2(\varphi)=\langle
F_2,\varphi\rangle$ for all $\varphi\in W$. By these operators
defined above, we can write the first  equation of \dref{11} as
\begin{equation}\label{15g5}
\dfrac{dw}{dt}+A_{k}w+C(z,w)=F_2.
\end{equation}
Combining \dref{15g1} and \dref{15g5}, we can write \dref{11} in the
abstract evolution equation as follows:
\begin{equation}\label{16}
\left\{\begin{array}{l}
 \dfrac{dz}{dt}+A_{\gamma}z+B(z)+\beta g w=F_1,\\
\dfrac{dw}{dt}+A_{k}w+C(z,w)=F_2,\\
z(0)=z_0, w(0)=w_0.
\end{array}\right.
\end{equation}

 \begin{lemma}\label{Le6}
 If $z\in L^2(0,T;V)$, then $B(z)\in L^1(0,T;V^*)$;  and if $w\in
 L^2(0,T;W)$, then $C(z,w)\in L^1(0,T;W^*)$.
 \end{lemma}
\begin{proof}  Apply the H\"{o}lder inequality and the Sobolev
embedding theorem $H^1(\Omega)\subset L^4(\Omega)$ to obtain
$$
|(B(z),\phi)|=|b(z,z,\phi)|\le
c_6\|z\|_1\|z\|_{L^4}\|\phi\|_{L^4}\le c_6\|z\|_1^2\|\phi\|_{1}
$$
and hence $\|B(z)\|_{V^*}\le c_6 \|z\|^2_1$ which
  shows that $B(z)\in L^1(0,T;V^*)$. Similarly, we have
$$
|(C(z,w),\varphi)|=|-c(z,\varphi,w)|\le
c_7\|z\|_{L^4}\|\varphi\|\|w\|_{L^4}\le
c_8\|z\|\|\varphi|\|w\|\le
c_9(\|z\|^2+\|w\|^2)\|\varphi\|
$$
and hence $\|C(z,w)\|_{W^*}\le c_9(\|z\|_{1}^2+\|w\|_{1}^2)$ which
shows that $C(z,w)\in L^1(0,T;W^*)$ for all $\varphi\in W$. The
proof is complete.
\end{proof}
We specified the constants $c_i,i=1,2,\cdots,9$ are in
this section in the remaining part of the paper.  The following
Lemma \ref{Le7} comes from theorem 2.2 of \cite{[9]} on page 220.

\begin{lemma}\label{Le7}
Let $X_0, X, X_1$ be Hilbert spaces with the compact embedding
relations
$$
X_0 \hookrightarrow X \hookrightarrow X_1.
$$
The for any bounded set $K\subset \R, \nu>0$, the embedding
$H_K^\nu(\R;X_0,X_1)\subset L^2(\R;X)$ is compact, where
$$
\begin{array}{l}
H_K^\nu(\R;X_0,X_1)=\{v\in L^2(\R;X_0)|\; {\rm supp}(v)\subset K,
\hat{p}_t^\nu v\in L^2(\R;X_1)\}, \crr\disp  \hat{p}_t^\nu
v(\tau)=(2\pi i\tau)^\nu \hat{v}(\tau), \; \hat{v}(\tau)
=\int_{-\infty}^\infty e^{-2\pi i t} v(t)dt, \crr\disp
\|v\|_{H_K^\nu(\R;X_0,X_1)}=\left[\|v\|^2_{L^2(\R;X_0)}+\||\tau|^2\hat{v}\|^2\right].
\end{array}
$$
\end{lemma}

\section{Existence of the weak solution}

This section discusses the existence of the weak solution defined by
Definition \dref{De1} to Equation \dref{1}. The main idea is to
construct to a Galerkin approximation.

 Choose two orthogonal bases $\{u_j\}_{j=1}^\infty $ for $ V$ and $\{\mu_j\}_{j=1}^\infty$ for
$W$ respectively.  Construct the Galerkin approximation solutions
\begin{equation}\label{guobz}
 z_m(x,t)=\sum_{j=1}^m q_{im}(t)u_j(x), \;
w_m(x,t)=\sum_{j=1}^m h_{im}(t)\mu_j(x)
\end{equation}
 such that for all $j\in
\N^+$, $\{z_m,w_m\}$ satisfies
\begin{equation}\label{20}
\left\{\begin{array}{l}
(z_m^\prime(t),u_j)+a_{\gamma(w_m)}(z_m(t),u_j)+b(z_m(t),z_m(t),u_j)+\beta
(w_m(t)g,u_j)=\langle f_1(t),u_j\rangle\\\hspace{11.5cm}+\langle
v_1,u_{j_n}\rangle_{\Gamma_1}, \\
(w_m^\prime(t),\mu_j)+a_{k(w_m)}(w_m(t),\mu_j)+c(z_m(t),w_m(t),\mu_j)
=\langle f_2(t),\mu_j\rangle+\langle v_2,\mu_j\rangle_{\Gamma_2},
\\z_m(0)=z_{m0}\to z_0 \hbox{ in }H, w_m(0)=w_{m0}\to w_0 \hbox{ in
}\tilde{H}, j=1,2\cdots,m.
\end{array}\right.
\end{equation}
where $z_{m0}$ is, for example, the orthogonal projection in $H$ of $z_0$ on the space spanned by ${u_1,u_2,\cdots,u_m}$ and $w_{m0}$ is the orthogonal projection in $\tilde{H}$ of $w_0$ on the space spanned by ${\mu_1,\mu_2,\cdots,\mu_m}$.

Hence once again, we write $z_m(t)=z_m(\cdot,t),
w_m(t)=w_m(\cdot,t)$, $f_i(t)=f_i(\cdot,t), v_i(t)=v_i(\cdot,t),
i=1,2$ by abuse of notation  without the confusion from the context.
It is seen that for any $m\in \N^+$, system \dref{20} is a system of
nonlinear differential equations with  the unknown variables
$\{q_{jm}(t),h_{jm}(t)\}$ and  the initial values
$q_{jm}(0)=(z_0,u_j), h_{jm}(0)=(w_0,\mu_j)$, $j=1,2,\cdots,m$. By
the assumption, this initial value problem admits a solution in some
interval $[0,t_m]$. We need to show that $t_m=T$.
\begin{lemma}\label{Le2.1} Let $\{z_m,w_m\}$ be the sequence
satisfying \dref{20}. Then
 there
exists a subsequence of $\{z_m,w_m\}$, still denoted by itself
without confusion, such that
\begin{equation}\label{32}
z_m \to z \hbox{ weakly in }L^2(0,T;V) \hbox{ and } z_m\to z \hbox{
weakly star in } L^\infty(0,\infty;H),
\end{equation}
where $z\in L^2(0,T;V)\cap L^\infty(0,T;H)$,  and
\begin{equation}\label{34}
w_m \to w \hbox{ weakly in }L^2(0,T;W) \hbox{ and } w_m\to w \hbox{
weakly star in } L^\infty(0,\infty;\tilde{H}),
\end{equation}
where $w\in L^2(0,T;W)\cap L^\infty(0,T;\tilde{H})$.
\end{lemma}
\begin{proof}
 By Lemmas \ref{Le2} \ref{Le3}, we
have
\begin{equation}\label{24}
 \left\{\begin{array}{l}
(z_m^\prime(t),z_m(t))+a_{\gamma(w_m)}(z_m(t),z_m(t))+\beta
(w_m(t)g,z_m(t))=\langle f_1(t),z_m(t)\rangle\\ \hspace{9.5cm}
+\langle
v_1(t),z_{m_n}(t)\rangle_{\Gamma_1}, \\
(w_m^\prime(t),w_m(t))+a_{k(w_m)}(w_m(t),w_m(t)) =\langle
f_2(t),w_m(t)\rangle+\langle v_2(t),w_m(t)\rangle_{\Gamma_2}.
\end{array}\right.
\end{equation}
 By assumption \dref{Assum}, for any given $\varepsilon>0$, we can
get from \dref{24} that
\begin{equation}\label{25g1}
 \begin{array}{ll}
\disp\dfrac{d}{dt}|z_m(t)|^2+2\gamma_0\|z_m(t)\|^2 & \le\disp
-2\beta (w_m(t)g,z_m(t))+2\langle f_1(t),z_m(t)\rangle\crr &+\disp
2\langle v_1(t),z_{m_n}(t)\rangle_{\Gamma_1}\le \gamma_0
\|z_m(t)\|^2\crr &+\disp
\dfrac{1}{\gamma_0}\|f_1(t)\|^2_{V^*}+\beta\|g\|_{\infty}\left(\dfrac{1}{\varepsilon}\|w_m(t)\|^2+
\varepsilon^2\|z_{m_n}(t)\|^2\right)\crr &+\disp
\varepsilon^2\|z_{m_n}(t)\|^2_{H^{1/2}(\Gamma_1)}
+\dfrac{1}{\varepsilon^2}\|v_1(t)\|^2_{H^{-1/2}(\Gamma_1)}.
\end{array}
\end{equation}
Here and hereafter, ${\|z_nm(t)\|_{[H^{1/2}(\Gamma_1)]^N}^2}$ we denote by ${\|z_nm(t)\|_{H^{1/2}(\Gamma_1]}^2}$ and
${\|v_1(t)\|_{[H^{-1/2}(\Gamma_1)]^N}^2}$
by ${\|v_1(t)\|_{H^{-1/2}(\Gamma_1)}^2}$  simply.

By the trace theorem from $H^1(\Omega)$ to $H^{1/2}(\Gamma)$, there
exists a constant $c_{10}>0$ such that
$$
\|z_{m_n}(t)\|_{H^{1/2}(\Gamma_1)}\le c_{10}\|z_{m}(t)\|.
$$
By substituting above inequality into \dref{25g1}, we obtain such as;
\begin{equation}\label{25g2}
 \begin{array}{l}
\disp\dfrac{d}{dt}|z_m(t)|^2+[\gamma_0-(\beta\|g\|_{\infty}\varepsilon^2+c_{10}\varepsilon^2)]\|z_m(t)\|^2\crr
\disp \hspace{1cm}  \le
\dfrac{1}{\gamma_0}\|f_1(t)\|^2_{V^*}+\beta\|g\|_{\infty}
\dfrac{1}{\varepsilon^2}\|w_m(t)\|^2
 +\dfrac{1}{\varepsilon^2}\|v_1(t)\|^2_{H^{-1/2}(\Gamma_1)}.
\end{array}
\end{equation}
Setting $\varepsilon^2=\gamma_0/(2(\beta\|g\|_{\infty}+c_{10})$ in
\dref{25g2} gives
\begin{equation}\label{25g3}
 \begin{array}{ll}
\disp\dfrac{d}{dt}|z_m(t)|^2+\dfrac{\gamma_0}{2}\|z_m(t)\|^2 &\le
\disp  \dfrac{1}{\gamma_0{c_1}}\|f_1(t)\|^2_{V^*} +\dfrac{2(\beta
\|g\|_{\infty}+c_{10})}{\gamma_0{c_1}}\beta \|g\|_{\infty}\|
\|w_m(t)\|^2\crr &+\disp \dfrac{2(\beta \|g\|_{\infty}+
c_{10})}{\gamma_0{c_1}} \|v_1(t)\|^2_{H^{-1/2}(\Gamma_1)}.
\end{array}
\end{equation}
By assumption \dref{Assum} again, for any given $\varepsilon>0$, we
can get from \dref{24} that
\begin{equation}\label{25g4}
 \begin{array}{ll}
\disp\dfrac{d}{dt}|w_m(t)|^2+2k_0\|w_m(t)\|^2  &\le \disp  k_0
\|w_m(t)\|^2+ \dfrac{1}{k_0}\|f_2(t)\|^2_{W^*}\crr &+\disp
\varepsilon^2\|w_{m}(t)\|^2 _{H^{1/2}(\Gamma_2)}
+\dfrac{1}{\varepsilon^2}\|v_2(t)\|^2_{H^{-1/2}(\Gamma_2)}.
\end{array}
\end{equation}
By the trace theorem from $H^1(\Omega)$ to $H^{1/2}(\Gamma)$ again,
there exists a constant $c_{11}>0$ such that
$$
\|w_{m_n}(t)\|_{H^{1/2}(\Gamma_2)}\le c_{11}\|w_m(t)\|.
$$
By substituting above inequality into \dref{25g4}, we obtain such as;
\begin{equation}\label{25g5}
\dfrac{d}{dt}|w_m(t)|^2+(k_0-c_{11}\varepsilon^2)\|w_m(t)\|^2
 \le \dfrac{1}{k_0}\|f_2(t)\|^2_{W^*}
 +\dfrac{1}{\varepsilon^2}\|v_2(t)\|^2_{H^{-1/2}(\Gamma_2)}.
\end{equation}
Setting $\varepsilon^2=k_0/(2c_{11})$ in \dref{25g5} gives
\begin{equation}\label{25g6}
\dfrac{d}{dt}|w_m(t)|^2+\dfrac{k_0}{2}\|w_m(t)\|^2
 \le \dfrac{1}{k_0}\|f_2(t)\|^2_{W^*}
 +\dfrac{2c_{11}}{k_0}\|v_2(t)\|^2_{H^{-1/2}(\Gamma_2)}.
\end{equation}
Integrate \dref{25g6} over $[0,T]$ with respect to $t$ to give
\begin{equation}\label{25g7}
\begin{array}{l}
\disp |w_m(T)|^2+\dfrac{k_0}{2}\int_0^T\|w_m(t)\|^2dt
 \le \dfrac{1}{k_0}\int_0^T\|f_2(t)\|^2_{W^*}dt
 +\dfrac{2c_{11}}{k_0}\int_0^T\|v_2(t)\|^2_{H^{-1/2}(\Gamma_2)}dt\crr\disp
 +|w_m(0)|^2
\disp  \le
 \dfrac{1}{k_0}\int_0^T\|f_2(t)\|^2_{W^*}dt
 +\dfrac{2c_{11}}{k_0}\int_0^T\|v_2(t)\|^2_{H^{-1/2}(\Gamma_2)}dt+|w(0)|^2.
 \end{array}
\end{equation}
Since the right-hand side of \dref{25g7} is bounded, we have
\begin{equation}\label{28}
\{w_m\} \hbox{ is a bounded sequence in }L^2(0,T;W).
\end{equation}
Replace $T$ by $t\in [0,T]$ in  \dref{25g7} to obtain
\begin{equation}\label{25g8}
ess\sup\limits_t|w_m(t)|^2 \le
 \dfrac{1}{k_0}\int_0^T\|f_2(t)\|^2_{W^*}dt
 +\dfrac{2c_{11}}{\varepsilon^2}\int_0^T\|v_2(t)\|^2_{H^{-1/2}(\Gamma_2)}dt+|w(0)|^2.
\end{equation}
Hence
\begin{equation}\label{29}
\{w_m\} \hbox{ is a bounded sequence in }L^\infty(0,T;L^2(\Omega)).
\end{equation}
On the other hand, integrate \dref{25g3} over $[0,T]$ with respect
to $t$  to give
\begin{equation}\label{25g10}
 \begin{array}{ll}
\disp|z_m(T)|^2+\dfrac{\gamma_0}{2}\int_0^T\|z_m(t)\|^2dt &\le \disp
\dfrac{1}{\gamma_0}\int_0^T\|f_1(t)\|^2_{V^*}dt \crr &+\disp
\dfrac{2(\beta \|g\|_{\infty}+c_{10})}{\gamma_0}\beta
\|g\|_{\infty}\| \int_0^T\|w_m(t)\|^2dt\crr &+\disp \dfrac{2(\beta
\|g\|_{\infty}+ c_{10})}{\gamma_0}
\int_0^T\|v_1(t)\|^2_{H^{-1/2}(\Gamma_1)}dt+|z_0|^2.
\end{array}
\end{equation}
Therefore
\begin{equation}\label{30}
\{z_m\} \hbox{ is a bounded sequence in }L^2(0,T;V).
\end{equation}
Replace $T$ by $t\in [0,T]$ in \dref{25g10} to get
\begin{equation}\label{25g11}
\begin{array}{ll}
 ess\sup\limits_t|z_m(t)|^2 &\le \disp
\dfrac{1}{\gamma_0}\int_0^T\|f_1(t)\|^2_{V^*}dt +\dfrac{2(\beta
\|g\|_{\infty}+c_{10})}{\gamma_0}\beta \|g\|_{\infty}\|
\int_0^T\|w_m(t)\|^2dt\crr &+\disp \dfrac{2(\beta \|g\|_{\infty}+
c_{10})}{\gamma_0}
\int_0^T\|v_1(t)\|^2_{H^{-1/2}(\Gamma_1)}dt+|z_0|^2.
\end{array}
\end{equation}
Therefore,
\begin{equation}\label{31}
\{z_m\} \hbox{ is a bounded sequence in }L^\infty(0,T;H).
\end{equation}
\dref{32} and \dref{34} then follow from  \dref{28} , \dref{30}
\dref{29} and \dref{31}.
\end{proof}

\begin{lemma}\label{Le2.2} Let $\{z_m,w_m\}$ be the sequence
determined by Lemma \ref{Le2.1}. Then there exists a sequence of
$\{z_m,w_m\}$, still denoted by itself without confusion, such that
\begin{equation}\label{45}
z_m \to z \hbox{ strongly in } L^2(0,T;H), \;w_m \to w \hbox{
strongly in } L^2(0,T;\tilde{H}).
\end{equation}
\end{lemma}
\begin{proof}
 By virtue of Lemmas
\ref{Le1}-\ref{Le3}, we can write \dref{20} as follows:
\begin{equation}\label{36}
\left\{\begin{array}{l} \disp  \left \langle
\dfrac{dz_m(t)}{dt},u_j\right\rangle =\langle F_1-\beta g
w_m(t)-A_{\gamma}z_m(t)-B(z_m(t),u_j\rangle, \forall\;
j=1,2,\cdots,m, \crr\disp \left\langle
\dfrac{dw_m(t)}{dt},\mu\right\rangle =\langle F_2- C(z_m(t),w_m(t))-A_k
w_m(t), \mu_j\rangle, \forall \; j=1,2,\cdots,m.
\end{array}\right.
\end{equation}
Denote by $\{\tilde{z}_m,\tilde{w}_m\}$ the $\{z_m,w_m\}$ with zero
values outside of $[0,T]$ and $\{\hat{z}_m,\hat{w}_m\}$ the
 Fourier transformations of $\{\tilde{z}_m,\tilde{w}_m\}$. We claim
 that there exists a $\nu>0$ such that
\begin{equation}\label{39}
\int_{-\infty}^\infty |\tau|^{2\nu}\|\hat{z}_m(\tau)\|^2 d\tau
<\infty.
\end{equation}
To this end, we write the first equation of \dref{36} as
\begin{equation}\label{41}
\dfrac{d}{dt}(\tilde{z}_m,u_j) =\langle
\tilde{f}_{1m},u_j\rangle+(z_{0m},u_j)\delta_0-(z_m(T),u_j)\delta_T,
\end{equation}
where $\delta_0,\delta_T$ are Dirac functions, and
$$
\begin{array}{l}
\disp \tilde{f}_{1m}(t) =f_{1m}(t) \hbox{ for }t\in [0,T] \hbox{ and
}\tilde{f}_{1m}(t) \hbox{ for }t>T, \crr\disp
 f_{1m}(t)=F_1-\beta
g w_m(t)-A_\gamma z_m(t)-B(z_m(t)).
\end{array}
$$
Take Fourier transform for Equation \dref{41} to get
\begin{equation}\label{42}
2\pi i\tau(\hat{z}_m,u_j) =\langle
\hat{f}_{1m},u_j\rangle+(z_{0m},u_j)-(z_m(T),u_j)e^{-2\pi i T\tau},
\end{equation}
where $\hat{f}_{1m}$ is the Fourier transform of $\tilde{f}_{1m}$.

  Let $\tilde{q}_{jm}(t)$ be the function of $q_{im}$ in $z_m(t)=\sum_{j=1}^m q_{jm}(t)u_j$ that is zero
  outside of $[0,T]$ and let $\hat{q}_{jm}(\tau)$ be its Fourier
  transform. Multiplier Equation \dref{42} by $\hat{q}_{jm}$ and sum
  for $j$ from 1 to $m$ to obtain
\begin{equation}\label{43}
2\pi i|\hat{z}_m(\tau)|^2=\langle
\hat{f}_{1m}(\tau),\hat{z}_m(\tau)\rangle+(z_{0m}(0),\hat{z}_m(\tau))-(z_m(T),\hat{z}_{m}(\tau))e^{-2\pi
i T\tau}.
\end{equation}
We thus conclude that
\begin{equation}\label{guo1}
\int_0^T\|f_{1m}(t)\|_{V^*}^2dt \le
\int_0^T\left[\|f_1(t)\|_{V^*}+c_2\|w_m(t)\|+\gamma_1\|z_m(t)\|+c_1\|z_{m}(t)\|^2\right]dt,
\end{equation}
where we used the fact that $\|Bz_m(t)\|\le c_1\|z_m(t)\|^2$. By
\dref{28}, \dref{29} and \dref{30}, it follows from \dref{guo1} that
\begin{equation}\label{guo2}
\sup_{\tau\in \R}\|\hat{f}_{1m}(\tau)\|_{V^*} <\infty.
\end{equation}
Apply \dref{guo2} and the facts $\sup_{m\in
\Z^+}[\|z_m(0)\|+\|z_m(T)\|]<\infty$ to \dref{43} to yield
\begin{equation}\label{44}
|\tau||\hat{z}_m(\tau)|^2 \le
c\acute{}_3\|\hat{z}_m(\tau)\|+c\acute{}_4|\hat{z}_m(\tau)|=c\acute{}_5\|\hat{z}_m(\tau)\|,
c\acute{}_5=c\acute{}_3+c\acute{}_4.
\end{equation}

For $\nu$ , fixed ,$\nu<1/4$ , we can observe that
$$
|\tau|^{2\nu} \le c'(\nu)\dfrac{1+|\tau|}{1+|\tau|^{1-2\nu}}, \forall\;
\tau \in \R
$$

From this inequality, we obtain
\begin{equation}\label{gbz1}
\int_{-\infty}^\infty |\tau|^{2\nu}|\hat{z}_m(\tau)|^2d\tau \le
c'(\nu)\int_{-\infty}^\infty
\dfrac{1+|\tau|}{1+|\tau|^{1-2\nu}}|\hat{z}_m(\tau)|^2d\tau
\end{equation}

Thus by \dref{44}, we can obtain the constants $c'_0>0$ and $c'_1>0$ such that
$$
\int_{-\infty}^\infty |\tau|^{2\nu}|\hat{z}_m(\tau)|^2d\tau \le
c'(\nu)\int_{-\infty}^\infty
\dfrac{1+|\tau|}{1+|\tau|^{1-2\nu}}|\hat{z}_m(\tau)|^2d\tau
\le
$$
$$
\le c'_0\int_{-\infty}^\infty
\dfrac{\|\hat{z}_m(\tau)\|}{1+|\tau|^{1-2\nu}}d\tau+c'_1\int_{-\infty}^\infty
\|\hat{z}_m(\tau)\|^2d\tau
$$
By \dref{45} and the Parseval equality the last integral is bounded as $m\to \infty$ ; thus \dref{39} will be proved if we show that;
$$
\int_{-\infty}^\infty
\dfrac{\|\hat{z}_m(\tau)\|}{1+|\tau|^{1-2\nu}}d\tau\leq{Const}
$$
By the Schwarz inequality and the Parseval equality we can obtain

$$
\int_{-\infty}^\infty
\dfrac{\|\hat{z}_m(\tau)\|}{1+|\tau|^{1-2\nu}}d\tau\leq{\le (\int_{-\infty}^\infty
\dfrac{d\tau}{(1+|\tau|^{1-2\nu})^2})^{1/2}(\int_{0}^T
\|\hat{z}_m(t)\|^2dt})^{1/2}
$$
which is finite since $\nu<1/4$ , and bounded as $m\to \infty$ ; by \dref{31}.
The proof of \dref{39} is achieved.
By \dref{39}, \dref{30} and \dref{31}, we conclude that
\begin{equation}\label{40}
\{z_m\} \hbox{ is bounded in } H^{\nu}(\R;V)\cap H^{\nu}(\R;H).
\end{equation}
By Lemma \ref{Le7}, there exists a subsequence of $\{z_m,w_m\}$ that
is still denoted by itself without confusion such that
 $$
z_m \to z \hbox{ strongly in } L^2(0,T;H), \;w_m \to w \hbox{
strongly in } L^2(0,T;L^2(\Omega)).
$$
This is \dref{45}.
\end{proof}

\begin{theorem}\label{Th1}
Suppose that the functions $\gamma,k$ satisfy that when $w_m\to w$
in $L^2(0,T;L^2(\Omega))$, then $\gamma(w_m)\to \gamma(w),
k(w_m)\to k(m)$ in $L^2(0,T;L^2(\Omega))$, respectively. Then there exists a weak solution to
\dref{1}.
\end{theorem}
\begin{proof}  Let $\Psi$ and $\theta$  be
continuous differentiable vector functions defined on $[0,T]$ with
$\Psi(T)=\theta(T)=0$. Multiply the first equation of \dref{20} by
$\Psi$ and integrate over $[0,T]$ with respect to $t$ to give
\begin{equation}\label{47}
\begin{array}{l}
\disp
-\int_0^T(z_m(t),\Psi'(t)u_j)dt+\int_0^T[a_{\gamma(w_m)}(z_m(t),\Psi(t)u_j)+b(z_m(t),z_m(t),\Psi(t)u_j)\crr\disp
+\beta (w_m(t)g,\Psi(t)u_j)]dt=(z_{0m},u_j)\Psi(0)+\int_0^T\langle
f_1(t),u_j\rangle dt+\int_0^T\langle
v_1,\Psi(t)u_{j_n}\rangle_{\Gamma_1}dt.
\end{array}
\end{equation}
 Multiply
the second  equation of \dref{20} by $\theta$ and integrate over
$[0,T]$ with respect to $t$ to give
\begin{equation}\label{48}
\begin{array}{l}
\disp
-\int_0^T(w_m(t),\theta'(t)\mu_j)dt+\int_0^T[a_{k(w_m)}(w_m(t),\theta(t)\mu_j)+c(z_m(t),w_m(t),\theta(t)\mu_j)\crr
\disp]dt =(w_{0m},\mu_j)\theta(0)+\int_0^T\langle
f_2(t),\theta(t)\mu_j\rangle dt+\int_0^T\langle
v_2,\theta(t)\mu_j\rangle_{\Gamma_2}dt
\end{array}
\end{equation}
Passing to the limit as $m\to \infty$ in \dref{47} and \dref{48} by
applying \dref{32}, \dref{45},  the properties (iii) in Lemmas
\ref{Le2} and \ref{Le3} for $b$ and $c$, and the continuous
assumption on $\gamma$ and $k$, we obtain
\begin{equation}\label{49}
\begin{array}{l}
\disp
-\int_0^T(z(t),\Psi'(t)u_j)dt+\int_0^T[a_{\gamma(w)}(z(t),\Psi(t)u_j)dt+b(z(t),z(t),\Psi(t)u_j)\crr\disp
+\beta (w(t)g,\Psi(t)u_j)]dt=(z_{0},u_j)\Psi(0)+\int_0^T\langle
f_1(t),u_j\rangle dt+\int_0^T\langle
v_1,\Psi(t)u_{j_n}\rangle_{\Gamma_1}dt.
\end{array}
\end{equation}
\begin{equation}\label{50}
\begin{array}{l}
\disp
-\int_0^T(w(t),\theta'(t)\mu_j)dt+\int_0^T[a_{k(w)}(w(t),\theta(t)\mu_j)+c(z(t),w(t),\theta(t)\mu_j)]dt\crr
\disp =(w_{0},\mu_j)\theta(0)+\int_0^T\langle
f_2(t),\theta(t)\mu_j\rangle dt+\int_0^T\langle
v_2,\theta(t)\mu_j\rangle_{\Gamma_2}dt,
\end{array}
\end{equation}
where in obtaining \dref{49} and \dref{50}, we used the following
facts:
\begin{itemize}

\item The convergence of the nonlinear terms in $b(\cdot,\cdot,\cdot)$ and
$c(\cdot,\cdot,\cdot)$ can be obtained  in the same way as that in
Chapter 3 of \cite{[9]}.

\item
$$
\begin{array}{ll}
\disp \int_0^Ta_{\gamma(w_m)}(z_m(t),\Psi(t)\mu_j)dt &= \disp
\int_0^T(\gamma(w_m)rot z_m(t),\Psi(t) rot u_j)dt\crr & =\disp
\int_0^T(rot z_m(t),\gamma(w_m)rot \Psi(t) rot u_j)dt\crr
&\to \disp \int_0^T(rot z(t),\gamma(w)\Psi(t) rot
u_j)dt=\int_0^T(\gamma(w)rot z(t),\Psi(t) rot u_j)dt \crr &=
\disp \int_0^Ta_{\gamma(w)}(z(t),\Psi(t)\mu_j)dt,
\end{array}
$$
where we used the facts that $rot z_m\to rot z$ weakly in
$L^2(0,T;V)$ and $\gamma(w_m)\Psi(t)rot u_j\to
\gamma(w)\Psi(t)\nabla u_j$ strongly in $L^2(0,T;H)$
from the assumptions of Theorem \ref{Th1}.

\item Similarly
$$
 \int_0^Ta_{k(w_m)}(w_m(t),\theta(t)\mu_j)dt
\to \int_0^Ta_{\gamma(w)}(w(t),\theta(t)\mu_j)dt
$$
by the facts  again $\nabla w_m\to \nabla w$ weakly in $L^2(0,T;W)$
and $k(w_m)\theta(t)\nabla \mu_j\to k(w)\theta(t)\nabla \mu_j$ strongly
in $L^2(0,T;\tilde{H})$ from the assumptions of Theorem \ref{Th1}.
\end{itemize}

By the density arguments, we have that \dref{49} and \dref{50} hold
true for any $\phi\in V$ instead of $u_j$  and $\varphi \in W$
instead on $\mu_j$, respectively. That is,
\begin{equation}\label{51}
\begin{array}{l}
\disp
-\int_0^T(z(t),\Psi'(t)\phi)dt+\int_0^T[a_{\gamma(w)}(z(t),\Psi(t)\phi)+b(z(t),z(t),\Psi(t)\phi)\crr\disp
+\beta (w(t)g,\Psi(t)\phi)]dt=(z_{0},\phi)\Psi(0)+\int_0^T\langle
f_1(t),\phi\rangle dt+\int_0^T\langle
v_1,\Psi(t)\phi_n\rangle_{\Gamma_1}dt, \forall\; \phi\in V;
\end{array}
\end{equation}
\begin{equation}\label{52}
\begin{array}{l}
\disp
-\int_0^T(w(t),\theta'(t)\varphi)dt+\int_0^T[a_{k(w)}(w(t),\theta(t)\varphi)+c(z(t),w(t),\theta(t)\varphi)]dt\crr
\disp =(w_{0},\varphi)\theta(0)+\int_0^T\langle
f_2(t),\theta(t)\varphi\rangle dt+\int_0^T\langle
v_2,\theta(t)\varphi\rangle_{\Gamma_2}dt, \forall\; \varphi \in W.
\end{array}
\end{equation}
Now take $\Psi\in (D(0,T))^N$ in \dref{51} and $\theta\in D(0,T)$ in
\dref{52}. Then $\{z,w\}$ satisfies
\begin{equation}\label{gbz5}
\left\{\begin{array}{l}
(z',\phi)+a_{\gamma(w)}(z,\phi)+b(z,z,\phi)+(\beta w g,\phi)=\langle
f_1,\phi\rangle+\langle v_1,\phi_n\rangle_{\Gamma_1}, \forall\;
\phi\in
V, \\
(w',\varphi)+a_{k(w)}(w,\varphi)+c(z,w,\varphi) =\langle
f_2,\varphi\rangle+\langle v_2,\varphi\rangle_{\Gamma_2}, \forall\;
\varphi\in W.
\end{array}\right.
\end{equation}
This is the equations in \dref{11}. Finally, we determine the
initial value of $\{z,w\}$. Actually, multiply the first equation of
\dref{gbz5} and integrate over $[0,T]$ with respect to $t$ to get
\begin{equation}\label{53}
\begin{array}{l}
\disp
-\int_0^T(z(t),\Psi'(t)\phi)dt+\int_0^T[a_{\gamma(w)}(z(t),\Psi(t)\phi)dt+b(z(t),z(t),\Psi(t)\phi)\crr\disp
+\beta (w(t)g,\Psi(t)\phi)]dt=(z(0),\phi)\Psi(0)+\int_0^T\langle
f_1(t),\phi\rangle dt+\int_0^T\langle
v_1,\Psi(t)\phi_n\rangle_{\Gamma_1}dt, \forall\; \phi\in V.
\end{array}
\end{equation}
Subtract \dref{53} from \dref{51} to get $(z(0)-z_0,\phi)\Psi(0)=0$.
Take $\Psi$ so that $\Psi(0)=1$ we get $(z(0)-z_0,\phi)=0$ for all
$\phi\in V$. So $z(0)=z_0$. The similar arguments lead to
$w(0)=w_0$. The proof is complete.
\end{proof}

\section{Uniqueness of the  weak solution}

\begin{theorem}\label{Th.2}  Let  $\gamma, k: L^2(\Omega)\to L^\infty(\Omega)$ are  Lipschitz
continuous functions, that is, there are constants $h_1, h_2>0$ such
that
\begin{equation}\label{54}
\begin{array}{l}
\|\gamma(w_*)-\gamma(w_{**})\|_{L^\infty}\le l_1\|w_*)-w_{**}\|, t\in
[0,T] \hbox{ a.e.},  \forall \; w_*, w_{**}\in L^2(\Omega), \crr
\|k(w_*)-k(w_{**})\|_{L^\infty}\le l_2\|w_*)-w_{**}\|,  t\in [0,T]
\hbox{ a.e.},  \forall \; w_*, w_{**}\in  L^2(\Omega).
\end{array}
\end{equation}
Suppose that the weak solution $\{z,w\}$ of \dref{1} claimed by
Theorem \ref{Th1} satisfies
\begin{equation}\label{55}
4d\left(\dfrac{\|z(t)\|_{L^4}}{c_1}+\dfrac{d}{k_0c_1\acute{c}_1}\|w(t)\|^2_{L^4}\right)<\gamma_0,
\forall\; t\in [0,T] \hbox{ a.e.},
\end{equation}
where  $d>0$ is the constant in Sobolev inequality that
$\|f\|_{L^4(\Omega)}\le d\|f\|_{H^1(\Omega)}$ for all $f\in
H^1(\Omega)$, which depends on $N$ and $\Omega$, and the constants
$c_1$ and $\acute{c}_1$ are that in  Lemma \ref{Le1}.
Then the weak solution is unique.
\end{theorem}
\begin{proof} Suppose that we have two weak solutions
$\{z_*,w_{*}\}$, $\{z_{*},z_{**}\}$ to \dref{1}. Set $z=z_*-z_{**},
w=w_*-w_{**}$. Then by \dref{11}
\begin{equation}\label{add1}
\left\{\begin{array}{ll} \disp \dfrac{d}{dt} (z,\phi)&+\disp
(\gamma(w_*)\nabla z,\nabla\phi)+b(z_*,z_*,\phi)
-b(z_{**},z_{**},\phi) +(\beta w g,\phi)\crr &=\disp
((\gamma(w_*)-\gamma(w_{**}))\nabla z_{**},\nabla\phi), \forall\;
\phi\in V, \crr \dfrac{d}{dt} (w,\varphi)&+\disp (k(w_*)\nabla
w,\nabla\varphi)+c(z_*,w_*,\varphi) -c(z_{**},w_{**},\varphi)\crr
&=\disp -((k(w_*)-k(w_{**}))\nabla w_{**},\nabla\varphi), \forall\;
\varphi\in W.
\end{array}\right.
\end{equation}
Set $\phi=z, \varphi =w$ in \dref{add1} to get
\begin{equation}\label{add2}
\left\{\begin{array}{ll} \disp \dfrac{d}{dt} (z,z)&+\disp
(\gamma(w_*)\nabla z,\nabla z)+b(z_*,z_*,z) -b(z_{**},z_{**},z)
+(\beta w g,z)\crr &=(\disp ((\gamma(w_*)-\gamma(w_{**}))\nabla
z_{**},\nabla\phi), \crr \dfrac{d}{dt} (w,w)&+\disp (k(w_*)\nabla
w,\nabla w)+c(z_*,w_*,w) -c(z_{**},w_{**},w)\crr &=\disp
-((k(w_*)-k(w_{**}))\nabla w_{**},\nabla w).
\end{array}\right.
\end{equation}
This together with $b(z{**},z,z)=c(z_{**},w,w)=0$ gives
\begin{equation}\label{add3}
\left\{\begin{array}{l} \disp \dfrac{d}{dt}\|z(t)\|^2+
2\gamma_0\|\nabla z(t)\|^2 \le  2b(z,z,z_*)+2(\beta g
w,z)-2((\gamma(w_*)-\gamma(w_{**}))\nabla z_{**},\nabla z), \crr
\disp  \dfrac{d}{dt} \|w(t)\|^2+  2k_0\|\nabla w(t)\|^2 \le
2c(z,w,w_*) +2((k(w_*)-k(w_{**}))\nabla w_{**},\nabla w).
\end{array}\right.
\end{equation}
By the Lipschitz continuity \dref{54}, for any given
$\varepsilon>0$, it follows from the ``$z$ part'' of \dref{add3}
that
\begin{equation}\label{add4}
\left\{\begin{array}{l} \disp \dfrac{d}{dt}\|z(t)\|^2+
2\gamma_0\|\nabla z(t)\|^2 \le 2\|z\|_{L^4}\|\nabla
z\|\|z_*\|_{L^4}+2\beta \|g\|_\infty\|w\|\|z\|+
2\|\gamma(w_*)\crr\disp -\gamma(w_{**})\|_{L^\infty}\|\nabla
z_{**}\|\|\nabla z\| \crr \disp \le 2\|z\|_{L^4}\|\nabla
z\|\|z_*\|_{L^4}+2\beta \|g\|_\infty\|w\|\|z\|+
2l_1\|w_*-w_{**}\|_{L^2}\|\nabla z_{**}\|\|\nabla z\|\crr\disp
  \le 2\|z\|_{L^4}\|\nabla z\|\|z_*\|_{L^4}+2\beta
\|g\|_\infty\|w\|\|z\|+
2l_1\left(\dfrac{\gamma_0}{2\varepsilon^2c_1}\|w\|^2_{L^2}\|\|\nabla
z_{**}\|^2+\dfrac{\varepsilon^2c_1}{\gamma_0}\|\nabla z\|^2\right).
\end{array}\right.
\end{equation}
Putting $\varepsilon^2=\gamma_0^2/l_1$ in \dref{add4}, we obtain
\begin{equation}\label{57}
  \dfrac{d}{dt}\|z(t)\|^2+ 2\gamma_0\|\nabla
z(t)\|^2 \le
  2\|z\|_{L^4}\|\nabla z\|\|z_*\|_{L^4}+2\beta
\|g\|_\infty\|w\|\|z\|+ \dfrac{l_1}{2\gamma_0c_1}\|w\|^2_{L^2}\|\nabla
z_{**}\|.
\end{equation}
Similar arguments to the ``$w$ part'' of  \dref{add3}, we have
\begin{equation}\label{58}
  \dfrac{d}{dt}\|w(t)\|^2+ 2k_0\|\nabla
w(t)\|^2 \le
  2\|z\|_{L^4}\|\nabla w\|\|w_*\|_{L^4} + \dfrac{l_2}{2k_0\acute{c_1}}\|w\|^2_{L^2}\|\nabla
w_{**}\|^2.
\end{equation}
Sum \dref{57} and \dref{58} to get
\begin{equation}\label{add5}
\begin{array}{l}
\disp
 \dfrac{d}{dt}[\|z(t)\|^2+\|w(t)\|^2]+ \gamma_0\|\nabla
z(t)\|^2 +k_0\|\nabla w(t)\|^2 \crr\disp \le
  2\|z\|_{L^4}\|\nabla z\|\|z_*\|_{L^4}+2\beta
\|g\|_\infty\|w\|\|z\|+ \dfrac{l_1}{2\gamma_0c_1}\|w\|^2_{L^2}\|\nabla
z_{**}\|+2\|z\|_{L^4}\|\nabla w\|\|w_*\|_{L^4} \crr\disp +
\dfrac{l_2}{2k_0\acute{c_1}}\|w\|^2_{L^2}\|\nabla w_{**}\|^2.
\end{array}
\end{equation}
Apply the H\"{o}lder inequality and the Sobolev inequality to the
right-hand side of \dref{add5} to obtain, for any given
$\varepsilon_1>0$, that
\begin{equation}\label{add6}
\begin{array}{l}
\disp
 \dfrac{d}{dt}[\|z(t)\|^2+\|w(t)\|^2]+ \gamma_0\|\nabla
z(t)\|^2 +k_0\|\nabla w(t)\|^2 \crr\disp \le
  2d\|\nabla z\|^2\|z_*\|_{L^4}+2\beta
\|g\|_\infty\|w\|\|z\|+ \dfrac{l_1}{2\gamma_0c_1}\|w\|^2_{L^2}\|\nabla
z_{**}\|+2d\|\nabla z\|\|\nabla w\|\|w_*\|_{L^4} \crr\disp +
\dfrac{l_2}{2k_0\acute{c_1}}\|w\|^2_{L^2}\|\nabla w_{**}\|^2\crr \le

2d\|\nabla z\|^2\|z_*\|_{L^4}+2\beta \|g\|_\infty\|w\|\|z\|+
\dfrac{l_1}{2\gamma_0c_1}\|w\|^2_{L^2}\|\nabla z_{**}\|\crr\disp
+2d\left(\dfrac{\varepsilon_0^2}{2}\|\nabla w\|^2
+\dfrac{1}{\varepsilon_1^2}\|\nabla z\|^2\|w_*\|^2_{L^4}\right)  +
\dfrac{l_2}{2k_0\acute{c_1}}\|w\|^2_{L^2}\|\nabla w_{**}\|^2.
\end{array}
\end{equation}
where $d>0$ is the constant from Sobolev inequality that
$\|f\|_{L^4(\Omega)}\le d\|f\|_{H^1(\Omega)}$ for all $f\in
H^1(\Omega)$.

Set $\varepsilon_1^2=k_0\acute{c}_1/2d$ in \dref{add6} to get
   \begin{equation}\label{add7}
   \begin{array}{l}
   \disp
    \dfrac{d}{dt}[|z(t)|^2+|w(t)|^2]+ \gamma_0c_1\|
   z(t)\|^2 +\dfrac{k_0\acute{c}_1}{2}\|w(t)\|^2 \crr\disp \le
     2d\|z\|^2\|z_*\|_{L^4}+\beta
   \|g\|_\infty(|w|^2+|z|^2)+ \dfrac{l_1}{2\gamma_0c_1}|w|^2\|
   z_{**}\|+\dfrac{2d^2}{k_0\acute{c}_1}\|
   z\|^2\||w_*\|^2_{L^4}\crr\disp +
   \dfrac{l_2}{2k_0\acute{c}_1}|w|^2\|w_{**}\|^2\crr\disp  \le
   \left(2d\|z_*\|_{L^4}+\dfrac{2d^2}{k_0\acute{c}_1}\|w_*\|^2_{L^4}\right)\|
   z\|^2+\left(\beta\|g\|_\infty
    + \dfrac{l_1}{2\gamma_0^2c_1}\|
   z_{**}\|^2 +\dfrac{l_2}{2k_0\acute{c}_1}\|
   w_{**}\|^2\right)|w|^2\crr\disp +\beta\|g\|_\infty|z|^2\crr\disp \le
   \dfrac{\gamma_0c_1}{2}\|
   z\|^2+\left(\beta\|g\|_\infty+\dfrac{l_1}{\gamma_0c_1}\|
   z_{**}\|^2+ \dfrac{l_2}{2k_0\acute{c}_1}\|
   w_{**}\|^2\right)|w|^2+\beta\|g\|_\infty|z|^2.
   \end{array}
   \end{equation}
   In the first term of the last row of \dref{add7}, we used the
   assumption \dref{55}. From \dref{add7}, we conclude that
   $$
   \begin{array}{l}
   \disp \dfrac{d}{dt}[|z(t)|^2+|w(t)|^2]+ \dfrac{\gamma_0c_1}{2}\|
   z(t)\|^2 +\dfrac{k_0\acute{c}_1}{2}\| w(t)\|^2\crr \disp
   \hspace{1cm} \le\left(\beta\|g\|_\infty+\dfrac{l_1}{2\gamma_0c_1}\|
   z_{**}\|^2+ \dfrac{l_2}{2k_0\acute{c}_1}\|
   w_{**}\|^2\right)|w|^2+\beta\|g\|_\infty|z|^2,
   \end{array}
   $$
   from which we obtain
   \begin{equation}\label{add8}
   \dfrac{d}{dt}[|z(t)|^2+|w(t)|^2] \le
    M(t)|w(t)|^2+N|z(t)|^2
    \end{equation}
   or
   \begin{equation}\label{add9}
    \dfrac{d}{dt}[|z(t)|^2+|w(t)|^2] \le
    (M(t)+N)[|w(t)|^2+|z(t)|^2),
    \end{equation}
    where
    $$
    M(t)=\beta\|g\|_\infty +\dfrac{l_1}{2\gamma_0c_1}\| z_{**}(t)\|^2
    +\dfrac{l_2}{2k_0\acute{c}_1}\| w_{**}(t)\|^2, N=\beta\|g\|_\infty.
    $$
   Since $M(\cdot)+N$ is integrable in $[0,T]$ with respect to $t$, we
   obtain,  from \dref{add9},  that
   \begin{equation}\label{add10}
   \dfrac{d}{dt}\left\{e^{-(M+N)}[|z(t)|^2+|w(t)|^2]\right\} \le0,
   \forall \; t\in [0,T] \hbox{ a.e.}.
    \end{equation}
   This together with $z(0)=w(0)=0$ gives
   $$
   |z(t)|^2+|w(t)|^2\le 0,  \forall \; t\in [0,T] \hbox{ a.e.}.
   $$
   Therefore, $z_*=z_{**}, w_*=w_{**}$. The proof is complete.
   \end{proof}

   \begin{remark}\label{Re4.1} If we denote (see, e.g.,\cite{[8]})
 $$
    {\rm Re(t)}=\dfrac{4d}{\gamma_0c_1}\|z(t)\|_{L^4} \hbox{ (the Reynold  number)};
    {\rm Ra(t)} =\dfrac{4d^2\|w(t)\|_{L^4}^2}{\gamma_0k_0c_1\acute{c}_1} \hbox{(the eigh number)},
   $$.

Then, condition $(\ref{55})$ is reduced to condition:
$$Re(t)+Ra(t)<1$$
\end{remark}
\begin{remark}\label{3}
 Theorem \ref{Th.2} is a generalization of the results
of \cite{[8]}. In \cite{[8]},  the boundary condition for the
velocity of fluid is given by the standard boundary condition,  that
is, the homogeneous Dirichlet boundary condition, and the  boundary
condition for the temperature of fluid is given by the
non-homogeneous Dirichlet boundary condition.
\end{remark}

\end{document}